\documentclass[11pt]{amsart}
\usepackage{a4wide}
\usepackage{amsmath,amsfonts,amssymb,amsthm,amscd}
\usepackage{mathrsfs}
\usepackage{extpfeil}
\usepackage{stmaryrd}
\usepackage{url}
\usepackage{indentfirst}
\usepackage{enumerate}
\usepackage{graphicx}
\usepackage{appendix}
\usepackage[numbers,sort&compress]{natbib}
\usepackage{color}
\usepackage[colorlinks,linkcolor=blue,citecolor=green]{hyperref}

\newtheorem{theorem}{Theorem}[section]
\newtheorem{lemma}{Lemma}[section]

\newtheorem{proposition}{Proposition}[section]
\numberwithin{equation}{section}

\newcommand{\R}{\mathbb R}
\newcommand{\olOmega}{\overline\Omega}
\newcommand{\MA}{\operatorname{MA}}

\title[Uniqueness for the degenerate Monge-Amp\`ere Equation]{Uniqueness for the degenerate Monge-Amp\`ere Equation\\
on Arbitrary Bounded Convex Domains}

 \author[Y. Zhou]{Yang Zhou}

 \address[Yang Zhou]{The Institute of Mathematical Sciences, The Chinese University of Hong Kong, Shatin, NT, Hong Kong, China}
 \email{yzhou@ims.cuhk.edu.hk}

\date{}

\begin{document}

\begin{abstract}
Let $n\ge2$ and let $\Omega\subset\R^n$ be an arbitrary bounded open convex set.
The author prove that, for $p>n$, the Dirichlet problem
\[
  \det D^2u=(-u)^p\quad\text{in }\Omega,
  \qquad u=0\quad\text{on }\partial\Omega,
  \qquad u>0\quad\text{in }\Omega
\]
has at most one convex Alexandrov solution.
The proof is based on the affine behavior of the Monge--Amp\`ere energy and on a power-concavity property of the $L^{p+1}$ mass along the Legendre path connecting two solutions.
At the homogeneous exponent $p=n$, the same argument shows that any two nonzero solutions with the same coefficient are positive multiples of one another.
\end{abstract}
\maketitle



\section{Introduction}

Let $n\ge2$, let $\Omega\subset\R^n$ be a bounded convex domain, 
and let $u\in C(\olOmega)$ be a non-trivial convex Alexandrov solution of the zero-Dirichlet problem
\begin{equation}\label{eq:mp}
	\begin{split}
		\det D^2u &=(-u)^p\quad\text{in }\Omega, \\
		u &=0\quad\qquad \text{on } \partial\Omega.
	\end{split}
\end{equation}
The exponent $p=n$ is homogeneous:
if $u$ is a solution, so is every positive multiple of $u$.  
For $p\ne n$ the coefficient in front of $(-u)^p$ can instead be removed by a scaling of $u$.  
This elementary dichotomy separates the eigenvalue problem $p=n$ from the subcritical and supercritical regimes.

The variational theory of \eqref{eq:mp} begins with Lions' work \cite{Lions1985} on the Monge--Amp\`ere eigenvalue on smooth uniformly convex domains.  
Tso \cite{Tso1990} introduced the real Monge--Amp\`ere functional, proved existence for $0<p\ne n$, and established uniqueness in the subcritical range $0<p<n$.  Hartenstine \cite{Ha2009} obtained the subcritical existence result on bounded strictly convex domains.  Later, Le \cite{Le2018} extended the existence theory, the critical variational characterization, and uniqueness of the critical eigenfunction to arbitrary bounded convex domains.  
These developments are also discussed from the viewpoint of generalized Monge--Amp\`ere functionals by Tong and Yau \cite{TongYau2023}.

Existence in the supercritical range $p>n$ also follows from \cite{Le2018}, whereas uniqueness in this regime is substantially more difficult.  
Huang \cite{Huang2019} proved uniqueness of least-energy solutions in dimension two for smooth uniformly convex domains and obtained full uniqueness for exponents in a range just above $n$.  
Cheng, Huang, and Xu \cite{CHX2024} proved further two-dimensional uniqueness results under symmetry assumptions, without requiring uniform convexity of the domain.  
As summarized by Le \cite[Section~1.1]{Le2026}, unrestricted uniqueness for $p>n$ on a general convex domain was still open.  
In this paper, we give an affirmative answer for this open problem.

\begin{theorem}\label{thm:main}
Let $p>n$.  Suppose $u_0,u_1\in C(\olOmega)$ are non-trivial convex Alexandrov solutions to \eqref{eq:mp}.
Then $u_0=u_1$ on $\olOmega$.
\end{theorem}

At $p=n$, the natural problem contains an eigenvalue parameter,
\[
  \det D^2u=\lambda[\Omega](-u)^n,
  \qquad u|_{\partial\Omega}=0.
\]
Lions \cite{Lions1985}, for smooth uniformly convex domains, and Le
\cite{Le2018}, for general bounded convex domains, proved that the eigenvalue is unique and that the corresponding nonzero eigenfunction is unique up to a positive multiplicative constant.  
Since scaling does not change the coefficient when $p=n$, the coefficient-one equation \eqref{eq:mp} may not have a nonzero solution on an arbitrary domain.  
Our argument also gives the following uniqueness statement (and, more generally, the same conclusion for any fixed positive coefficient).

\begin{theorem}\label{thm:eigen}
Let $p=n$, and let $u_0,u_1$ satisfy all the other assumptions of
Theorem~\ref{thm:main}.  
Then $u_1=c u_0$ for some $c>0$.
\end{theorem}

The Alexandrov Monge--Amp\`ere measure of $u$ is the Borel measure determined by
\[
  \MA[u](E):=\bigl|\partial u(E)\bigr|
\]
for Borel sets $E\subset\Omega$.  
Here $\partial u(E)$ is the union of the subdifferentials $\partial u(x)$ over $x\in E$.  
We say that a convex function $u$ is an Alexandrov solution of $\det D^2u=f$ if $\MA[u]=f\,dx$; see Figalli's book \cite{Fi2017} for details.

The key idea is to analyze the Monge--Amp\`ere energy and the
$L^{p+1}$ mass along a Legendre path.  
For a convex function $u$ with zero boundary values, define
\[
  \mathcal E(u):=\int_\Omega -u\,d\MA[u].
\]
Motivated by the duality method in Rockafellar \cite[Theorem~16.4]{Rockafellar1970}, and by Salani's use of
infimal convolution and the affine energy identity
\cite[Proposition~14]{Salani2005} (see also \cite{Semmes1988,CorderoKlartag2012}), we proceed as follows.  
Extend a convex function $v$ on $\olOmega$ by $+\infty$ outside $\olOmega$ and define
\begin{equation}\label{eq:le-v}
  v^*(y):=\sup\{x\cdot y-v(x):x\in\olOmega\}.
\end{equation}
For $u_0,u_1\in C(\olOmega)$, their Legendre path is
\[
  u_t:=\bigl((1-t)u_0^*+tu_1^*\bigr)^*,\qquad 0\le t\le1.
\]
Along this path, we prove
\[
  \mathcal E(u_t)=(1-t)\mathcal E(u_0)+t\mathcal E(u_1).
\]
Moreover, for every integer $m\ge p+1$, the function
\[
  P(t):=\left(\int_\Omega |u_t(x)|^{p+1}\,dx\right)^{1/(n+m)}
\]
is concave on $[0,1]$.  
Note that $P$ is a fixed positive power of the
$L^{p+1}$ mass.  
The supporting inequality for the Monge--Amp\`ere energy forces both endpoint slopes of this concave function to vanish.  
Hence $P$ is constant, and its equality case shows that
the normalized solutions are equal.  
This yields uniqueness.  
The concavity and rigidity of $P$ follow from the classical Brunn--Minkowski inequality and its equality case.

Our proof is based on convex analysis and does not require strict convexity or smoothness of $\Omega$.  
When $\Omega$ is smooth and uniformly convex, global boundary regularity is available.  
At the critical exponent $p=n$, Le--Savin \cite{LeSavin2017} prove global smoothness of the Monge--Amp\`ere eigenfunctions.
For $p>0$ with $p\ne n$, Le \cite[Theorem~1.6]{Le2018} gives global $C^{2,\beta}$ regularity.    
In that smoother setting, a different proof of Theorem \ref{thm:main} based on PDE techniques and degree theory is being developed in the forthcoming manuscript \cite{Z2026}.

The paper is organized as follows.  Section~\ref{S2} establishes the dual
representation and affine behavior of the Monge--Amp\`ere energy along the
Legendre path.  Section~\ref{S3} proves the required $L^{p+1}$ concavity and
its rigidity.  The main theorems are proved in Section~\ref{S4}.

\vspace{5mm}

\noindent{\bf Acknowledgement.}
The author would like to thank Professor G. Huang for introducing this research topic and for providing timely and valuable feedback.

\vspace{3mm}
\noindent\textbf{Disclosure on AI assistance.}
The author used AI-assisted tools, principally ChatGPT.  The author verified
all theorem statements and proofs and takes full responsibility for the
contents of the paper.

\section{The Monge-Amp\`ere energy}\label{S2}

Throughout the paper, $\Omega\subset\R^n$ is a bounded convex domain.  
If $v\in C(\olOmega)$ is convex and vanishes on $\partial\Omega$, then its Legendre transform $v^*$, defined in \eqref{eq:le-v}, is finite and convex on $\R^n$, and hence differentiable almost everywhere.

We first record a Legendre push-forward identity; compare McCann \cite[Lemma~4.1]{McCann1997}.
	
\begin{lemma}\label{lem:push}
For every nonnegative Borel function $F$ on $\olOmega$ that vanishes on
$\partial\Omega$, it holds
\begin{equation}\label{eq:push}
  \int_{\R^n}F\bigl(\nabla v^*(y)\bigr)\,dy
  =\int_\Omega F(x)\,d\MA[v](x).
\end{equation}
\end{lemma}

\begin{proof}
Extend $v$ by $+\infty$ outside $\olOmega$.  
At every differentiability point $y$ of $v^*$, the maximizer in \eqref{eq:le-v} is unique and equals $x=\nabla v^*(y)$.  
Fenchel equality gives
\[
  y\in\partial v(x),\qquad v^*(y)+v(x)=x\cdot y,
\]
and therefore
\[
  v^*(y)-y\cdot\nabla v^*(y)
  =-v\bigl(\nabla v^*(y)\bigr)\ge0.
\]
Conversely, if $E\subset\Omega$ is Borel and $v^*$ is differentiable at $y$, then
\[
  \nabla v^*(y)\in E
  \quad\Longleftrightarrow\quad
  y\in\partial v(E).
\]
Indeed, for an interior point $x\in\Omega$, the subdifferential of the extended function agrees with the usual subdifferential of $v$ in $\Omega$.
Consequently,
\begin{equation}\label{eq:area}
  \bigl|\{y:\nabla v^*(y)\in E\}\bigr|
  =|\partial v(E)|=\MA[v](E).
\end{equation}
The identity \eqref{eq:push} follows first for nonnegative simple functions supported in $\Omega$ and then, by monotone convergence, for every admissible $F$ vanishing on $\partial\Omega$ in this lemma. 
\end{proof}

\subsection{The dual-representation Monge--Amp\`ere energy}

The aim of this subsection is to derive a boundary-free representation of the Monge--Amp\`ere energy on the Legendre side.  
Set $K:=\olOmega$ and let $h_K(y):=\sup_{x\in K}x\cdot y$ be the support function of $K$.

\begin{lemma}\label{lem:energy}
For a convex function $v\in C(\olOmega)$ vanishing on $\partial\Omega$, the following quantities are equal in $[0,+\infty]$:
\begin{equation}\label{eq:energy}
  \mathcal E(v)
  :=\int_\Omega -v\,d\MA[v]
  =\int_{\R^n}\bigl(v^*-y\cdot\nabla v^*\bigr)\,dy
  =(n+1)\int_{\R^n}\bigl(v^*-h_K\bigr)\,dy.
\end{equation}
In particular, one of the three expressions is finite if and only if all three are finite.
\end{lemma}

\begin{proof}
Convexity and the zero boundary values imply $v\le0$ on $K$, so $F=-v$ is
admissible in Lemma~\ref{lem:push}.  
The first equality then follows from that lemma and Fenchel equality.

Set
\[
  \delta_v(y):=v^*(y)-h_K(y).
\]
For $y\ne0$, choose $x_y\in K$ with $x_y\cdot y=h_K(y)$.  
Such a point lies on $\partial K$, and hence $v(x_y)=0$; 
therefore $v^*(y)\ge x_y\cdot y-v(x_y)=h_K(y)$.  
The same inequality at $y=0$ follows from $v\le0$ and $v^*\ge0$.  Thus $\delta_v\ge0$.

Writing $y=r\theta$, where $r=|y|$ and $\theta\in\mathbb S^{n-1}$, define $d_\theta(r):=\delta_v(r\theta)$.  
From the definitions,
\[
  d_\theta(r)=\sup_{x\in K}
  \bigl\{-v(x)-r\bigl(h_K(\theta)-\theta\cdot x\bigr)\bigr\}.
\]
For each fixed $x$, the displayed function of $r$ is affine with nonpositive slope.  
Its supremum is therefore convex and nonincreasing.
Moreover, it is nonnegative by $\delta_v\ge0$ and, being finite and convex, is locally absolutely continuous.  
Since $h_K$ is positively homogeneous, for almost every $r>0$ and almost every $\theta$,
\[
  v^*(r\theta)-r\theta\cdot\nabla v^*(r\theta)
  =d_\theta(r)-r d_\theta'(r)\ge d_\theta(r).
\]
Hence the finiteness of the middle integral in \eqref{eq:energy} implies $\delta_v\in L^1(\R^n)$.

For every $R>0$, radial integration by parts gives
\begin{align}\label{eq:parts}
  \int_{B_R}\bigl(v^*-y\cdot\nabla v^*\bigr)\,dy
  &=\int_{\mathbb S^{n-1}}\int_0^R
    \bigl(d_\theta-rd_\theta'\bigr)r^{n-1}\,dr\,d\theta \notag\\
  &=(n+1)\int_{B_R}\delta_v\,dy
    -R^n\int_{\mathbb S^{n-1}}d_\theta(R)\,d\theta.
\end{align}
If $\delta_v\in L^1(\R^n)$, monotonicity of $d_\theta$ gives
\[
  \int_{B_R\setminus B_{R/2}}\delta_v\,dy
  \ge \frac{1-2^{-n}}{n}R^n
       \int_{\mathbb S^{n-1}}d_\theta(R)\,d\theta.
\]
The left-hand side tends to zero as $R\to\infty$, so the boundary term in \eqref{eq:parts} tends to zero.  
Letting $R\to\infty$ proves the last equality in \eqref{eq:energy}.  Conversely, finiteness of the last integral means $\delta_v\in L^1(\R^n)$, and the same argument proves finiteness and equality
of the middle integral.  
Together with the first equality, this proves the lemma.
\end{proof}

We next derive a supporting inequality for the dual Monge--Amp\`ere energy.

\begin{proposition}\label{prop:support}
Assume that $v_0,v_1\in C(\olOmega)$ are convex, vanish on $\partial\Omega$, satisfy $\mathcal E(v_0),\mathcal E(v_1)<\infty$, and  $\MA[v_0](\Omega)<\infty$.  
Then
\begin{equation}\label{eq:support}
  \mathcal E(v_1)-\mathcal E(v_0)
  \ge (n+1)\int_\Omega(v_0-v_1)\,d\MA[v_0].
\end{equation}
\end{proposition}

\begin{proof}
At almost every $y$, set $x=\nabla v_0^*(y)$.  The definition of $v_1^*$ and Fenchel equality for $v_0^*$ give
\[
  v_1^*(y)\ge x\cdot y-v_1(x),
  \qquad v_0^*(y)=x\cdot y-v_0(x).
\]
Therefore
\begin{equation}\label{eq:point}
  v_1^*(y)-v_0^*(y)
  \ge v_0\bigl(\nabla v_0^*(y)\bigr)
     -v_1\bigl(\nabla v_0^*(y)\bigr).
\end{equation}
By Lemma~\ref{lem:energy}, both $\delta_{v_0}$ and $\delta_{v_1}$ belong to $L^1(\R^n)$, and
\[
  v_1^*-v_0^*=\delta_{v_1}-\delta_{v_0}\in L^1(\R^n).
\]
The right-hand side of \eqref{eq:point} is integrable by applying Lemma~\ref{lem:push} separately to $(v_0-v_1)_+$ and $(v_0-v_1)_-$, and using $\MA[v_0](\Omega)<\infty$.  
Therefore, integrating \eqref{eq:point} and invoking Lemmas~\ref{lem:push} and \ref{lem:energy} yields \eqref{eq:support}.
\end{proof}

\subsection{The Monge--Amp\`ere energy along the Legendre path}

Let $v_0,v_1\in C(\olOmega)$ be nonzero convex functions that vanish on $\partial\Omega$, and write $\phi_i=v_i^*$ for $i=0,1$.  
For $0\le t\le1$, define
\begin{equation}\label{eq:path}
  \phi_t:=(1-t)\phi_0+t\phi_1,
  \qquad v_t:=\phi_t^*.
\end{equation}
The following facts are Alexandrov-level analogues of properties used in
\cite{Salani2005}.

\begin{lemma}\label{lem:path}
The function $v_t$ is continuous and convex on $\olOmega$, negative in
$\Omega$, and zero on $\partial\Omega$.  More explicitly,
\begin{equation}\label{eq:sup}
  v_t(x)=
  \inf_{\substack{x=(1-t)x_0+t x_1\\x_0,x_1\in\olOmega}}
  \Bigl((1-t)v_0(x_0)+t v_1(x_1)\Bigr).
\end{equation}
Furthermore:
\begin{itemize}
  \item[(i)] whenever $\mathcal E(v_0),\mathcal E(v_1)<\infty$,
  \begin{equation}\label{eq:affine}
    \mathcal E(v_t)=(1-t)\mathcal E(v_0)+t\mathcal E(v_1);
  \end{equation}
  \item[(ii)] the path is uniformly Lipschitz in $t$:
  \begin{equation}\label{eq:lip}
    \|v_t-v_s\|_{L^\infty(K)}
    \le |t-s|\,\|v_1-v_0\|_{L^\infty(K)},
    \qquad s,t\in[0,1];
  \end{equation}
  \item[(iii)] if $t=(1-\tau)t_0+\tau t_1$, then $v_t$ is the weighted
  infimal convolution of $v_{t_0}$ and $v_{t_1}$ with weights $1-\tau$ and $\tau$, as in \eqref{eq:sup}.
\end{itemize}
\end{lemma}

\begin{proof}
For $0<t<1$, Rockafellar's conjugacy theorem
\cite[Theorem~16.4]{Rockafellar1970}, applied with the usual scaling rule for convex conjugates, gives \eqref{eq:sup}.  Since $\olOmega$ is compact, the infimum is attained.  
The formula also shows that the work domain of $v_t$ is $\olOmega$.  Taking $x_0=x_1=x$ gives $v_t\le(1-t)v_0+tv_1$, and hence $v_t<0$ in $\Omega$ and $v_t\le0$ on $\olOmega=:K$.

Let $x\in\partial\Omega$ and choose a supporting linear functional $\ell$ such that $\ell(x)=\max_K\ell$.  
If $0<t<1$ and $x=(1-t)x_0+tx_1$ with $x_0,x_1\in K$, then
\[
  \max_K\ell=\ell(x)=(1-t)\ell(x_0)+t\ell(x_1)
\]
forces $x_0,x_1$ to lie in the exposed face $\{\ell=\max_K\ell\}\subset\partial\Omega$.  
Since $v_0=v_1=0$ on the boundary, \eqref{eq:sup} yields $v_t(x)=0$.  The function $v_t$ is finite and convex in $\Omega$, so it is continuous there.  
If $x_k\in\Omega$ tends to $x\in\partial\Omega$, lower semicontinuity gives $0=v_t(x)\le\liminf_k v_t(x_k)$, while $v_t\le0$; hence $v_t(x_k)\to0$.
Thus $v_t\in C(K)$.

Because $\phi_t$ is finite, closed, and convex, $v_t^*=\phi_t$.  Therefore
\[
  \phi_t-h_K=(1-t)(\phi_0-h_K)+t(\phi_1-h_K),
\]
and \eqref{eq:affine} follows from Lemma~\ref{lem:energy}.

Set $M=\|v_1-v_0\|_{L^\infty(K)}$.  
Directly from the definition of the conjugate, $|\phi_1-\phi_0|\le M$ on $\R^n$, and hence $|\phi_t-\phi_s|\le|t-s|M$.  
Conjugation reverses order and preserves additive constants, so the same bound holds for $v_t-v_s$ on $K$.  
This proves \eqref{eq:lip}.

Finally,
\[
  \phi_{(1-\tau)t_0+\tau t_1}
  =(1-\tau)\phi_{t_0}+\tau\phi_{t_1}.
\]
Applying the same infimal-convolution identity, and using
$v_{t_i}^*=\phi_{t_i}$, proves the semigroup property (iii).
\end{proof}

\section{Power concavity and rigidity}\label{S3}

In this section, we prove the concavity and rigidity of a functional $P(t)$ on $[0,1]$ that is central to the main theorems.

\begin{proposition}\label{prop:conc}
Let $v_t$ be the path in \eqref{eq:path}.  
For $q>1$ and an integer $m\ge q$,
set
\[
  S(t):=\int_\Omega |v_t|^q\,dx,
  \qquad P(t):=S(t)^{1/(n+m)}.
\]
Then $P$ is concave on $[0,1]$.  
If $P$ is constant, then $v_0=v_1$.
\end{proposition}

For proving Proposition \ref{prop:conc}, we use the classical Brunn--Minkowski inequality with its equality case; see, for example, \cite[Sections~2 and~5]{Gardner2002}.

\begin{lemma}[Brunn--Minkowski]\label{lem:BM}
Let $A,B$ be compact convex bodies with nonempty interior in $\R^d$.  
For $0<t<1$,
\begin{equation}\label{eq:BM}
  |(1-t)A+tB|^{1/d}
  \ge(1-t)|A|^{1/d}+t|B|^{1/d}.
\end{equation}
Equality at one $t\in(0,1)$ holds if and only if $B=cA+b$ for some $c>0$ and $b\in\R^d$.
\end{lemma}

Let $q>1$, choose an integer $m\ge q$, and set
\[
  a:=\frac qm\in(0,1].
\]
For a nonzero nonpositive convex function $v$ on $K$ that vanishes on
$\partial K$, define
\begin{equation}\label{eq:body}
  \mathcal K[v]
  :=\{(x,z)\in K\times\R^m:|z|\le |v(x)|^a\}.
\end{equation}
Because $-v$ is concave and $s\mapsto s^a$ is increasing and concave, $|v|^a$ is concave.  
The triangle inequality then shows that $\mathcal K[v]$ is a compact convex body with nonempty interior.  
If $\omega_m$ is the volume of the unit ball in $\R^m$, Fubini's theorem gives
\begin{equation}\label{eq:volume}
  |\mathcal K[v]|=\omega_m\int_\Omega |v|^q\,dx.
\end{equation}

The dimensional-lifting argument below is motivated by the Brunn--Minkowski method for power-concave functions; compare \cite{BrascampLieb1976}.  
We include the proof because its equality case is essential.

\begin{proof}[Proof of Proposition~\ref{prop:conc}]
Fix $0\le t_0<t_1\le1$, $0\le\tau\le1$, and put $t=(1-\tau)t_0+\tau t_1$.  
Take $(x_i,z_i)\in\mathcal K[v_{t_i}]$ for $i=0,1$, and set $x=(1-\tau)x_0+\tau x_1$.  
The semigroup property in Lemma \ref{lem:path} and the concavity of $s^a$ give
\begin{align*}
 |v_t(x)|^a
 &\ge\bigl((1-\tau)|v_{t_0}(x_0)|
            +\tau |v_{t_1}(x_1)|\bigr)^a\\
 &\ge(1-\tau)|v_{t_0}(x_0)|^a
      +\tau |v_{t_1}(x_1)|^a\\
 &\ge |(1-\tau)z_0+\tau z_1|.
\end{align*}
Consequently,
\begin{equation}\label{eq:inclu}
  \mathcal K[v_t]\supset
  (1-\tau)\mathcal K[v_{t_0}]+\tau\mathcal K[v_{t_1}].
\end{equation}
The Brunn--Minkowski inequality \eqref{eq:BM}, together with \eqref{eq:volume}, shows that $P$ is concave.

Suppose now that $P$ is constant.  
Then
\begin{equation}\label{eq:equ-v}
	|\mathcal K[v_t]|=|\mathcal K[v_0]|=|\mathcal K[v_1]|
	\qquad\text{for every }t\in[0,1].
\end{equation}
For any fixed $0<t<1$, inclusion \eqref{eq:inclu} and Brunn--Minkowski give
\begin{align*}
 |\mathcal K[v_t]|
 &\ge |(1-t)\mathcal K[v_0]+t\mathcal K[v_1]|,\\
 |(1-t)\mathcal K[v_0]+t\mathcal K[v_1]|^{1/(n+m)}
 &\ge(1-t)|\mathcal K[v_0]|^{1/(n+m)}
       +t|\mathcal K[v_1]|^{1/(n+m)}.
\end{align*}
The first and last quantities are equal, so equality holds throughout and, in particular, in Brunn--Minkowski inequality.  
Lemma~\ref{lem:BM} yields $\mathcal K[v_1]=c\mathcal K[v_0]+b$ for some $c>0$ and $b\in\R^{n+m}$.  
Then \eqref{eq:equ-v} implies $c=1$.  
The projection of both bodies onto the first $n$ coordinates is $K$.  
Writing $b=(b_x,b_z)$, we obtain $K=K+b_x$, and boundedness forces $b_x=0$.  For every $x\in\Omega$, the fiber of each body over $x$ is a Euclidean ball centered at the origin.  
Then equality of this fiber deduces $b_z=0$.  
Thus $\mathcal K[v_0]=\mathcal K[v_1]$, which gives $|v_0(x)|^a=|v_1(x)|^a$ for all $x\in K$.  
Since both functions are nonpositive, $v_0=v_1$.
\end{proof}

\section{Proof of the main theorems}\label{S4}
We are now in a position to prove the main theorems.

\begin{proof}[Proof of Theorems~\ref{thm:main} and~\ref{thm:eigen}]
Fix either $p>n$ or $p=n$.  
Since $u_0,u_1\in C(\olOmega)$ are nonzero convex solutions of \eqref{eq:mp}, convexity and the zero boundary condition imply $u_i\le0$ on $\olOmega$ and  $u_i<0$ in $\Omega$.  
Moreover,
\[
  \MA[u_i]=(-u_i)^p\,dx,\qquad i=0,1.
\]
Lemma~\ref{lem:energy} gives
\begin{equation}\label{eq:solenergy}
  0<\mathcal E(u_i)=\int_\Omega(-u_i)^{p+1}\,dx<\infty.
\end{equation}
Set
\[
  c_i:=\mathcal E(u_i)^{-1/(n+1)},
  \qquad \widehat u_i:=c_i u_i.
\]
The homogeneity of the Monge--Amp\`ere measure yields
\begin{equation*}
  \mathcal E(\widehat u_i)=1,
  \qquad
  \MA[\widehat u_i]
  =\lambda_i(-\widehat u_i)^p\,dx,
  \qquad
  \lambda_i:=c_i^{n-p}>0.
\end{equation*}

Join $\widehat u_0$ to $\widehat u_1$ by the Legendre path $\widehat u_t$ defined in \eqref{eq:path}.  
By \eqref{eq:affine},
\begin{equation*}
  \mathcal E(\widehat u_t)=1\qquad(0\le t\le1).
\end{equation*}
Proposition~\ref{prop:support}, applied first with base point $\widehat u_0$ and then with base point $\widehat u_1$, gives
\begin{align}
 \int_\Omega-(\widehat u_t-\widehat u_0)(-\widehat u_0)^p\,dx
 &=\lambda_0^{-1}\int_\Omega-(\widehat u_t-\widehat u_0)
     \,d\MA[\widehat u_0]\le0, \label{eq:left} \\
 \int_\Omega-(\widehat u_t-\widehat u_1)(-\widehat u_1)^p\,dx
 &=\lambda_1^{-1}\int_\Omega-(\widehat u_t-\widehat u_1)
     \,d\MA[\widehat u_1]\le0. \label{eq:right}
\end{align}

Put $q=p+1$ and choose an integer $m\ge q$.  
In both cases under consideration, $q\ge2$.  Set
\[
  S(t):=\int_\Omega(-\widehat u_t)^q\,dx,
  \qquad P(t):=S(t)^{1/(n+m)}.
\]
By \eqref{eq:lip}, the nonnegative functions $-\widehat u_t$ take values in a
fixed interval $[0,M]$, where $M$ depends on $\|u_0\|_{L^\infty(K)}$ and $\|u_1\|_{L^\infty(K)}$.  
Since $q\ge2$, Taylor's theorem gives, uniformly for $a,b\in[0,M]$,
\[
  \bigl|b^q-a^q-qa^{q-1}(b-a)\bigr|
  \le C|b-a|^2.
\]
Using \eqref{eq:lip} with $s=0$ and integrating, we obtain, as $t\downarrow0$,
\begin{equation}\label{eq:taylor}
  S(t)-S(0)
  =(p+1)\int_\Omega-
       (\widehat u_t-\widehat u_0)(-\widehat u_0)^p\,dx+O(t^2).
\end{equation}
Together with \eqref{eq:left}, this gives
\begin{equation}\label{eq:d0}
  P'_+(0)\le0.
\end{equation}
Here we used $S(0)>0$ and the fact that $r\mapsto r^{1/(n+m)}$ is increasing
and $C^1$ near $S(0)$.  
The identical expansion at $t=1$, combined with \eqref{eq:right} and division by the negative number $t-1$, gives
\begin{equation}\label{eq:d1}
  P'_-(1)\ge0.
\end{equation}

The one-sided derivatives exist because $P$ is concave by Proposition~\ref{prop:conc}.  
By the concavity and \eqref{eq:d0}-\eqref{eq:d1}, we have
\[
  P'_+(0) = P'_-(1) =0.
\]
Hence $P$ is constant on $[0,1]$ by using the concavity of $P(t)$ again.  
Then, it follows from rigidity assertion in Proposition~\ref{prop:conc} that
\begin{equation}\label{eq:equal}
  \widehat u_0=\widehat u_1.
\end{equation}

Thus $u_1=Cu_0$ for some $C>0$.  Substitution into the Alexandrov equation gives
\[
  C^n(-u_0)^p\,dx
  =\MA[Cu_0]
  =(-Cu_0)^p\,dx
  =C^p(-u_0)^p\,dx.
\]
Since $u_0<0$ in $\Omega$, $C^n=C^p$.  
If $p>n$, then $C=1$, proving Theorem~\ref{thm:main}.  
If $p=n$, \eqref{eq:equal} already says that the two original solutions are positive multiples of one another, proving Theorem~\ref{thm:eigen}.
\end{proof}

\end{document}